\documentclass[11pt]{amsart}

\usepackage{amsmath}
\usepackage{fullpage}
\usepackage{xspace}
\usepackage[psamsfonts]{amssymb}
\usepackage[latin1]{inputenc}
\usepackage{graphicx,color}
\usepackage[curve]{xypic} 
\usepackage{hyperref}
\usepackage{graphicx}

\usepackage{amsmath}%
\usepackage{amsthm}%
\usepackage{amscd}
\usepackage{amsfonts}%
\usepackage{amssymb}%
\usepackage{graphicx}

\usepackage{mathrsfs}

\usepackage{tikz}
\usetikzlibrary{matrix,arrows}

\newtheorem{theorem}{Theorem}[section]

\newtheorem{corollary}[theorem]{Corollary}

\newtheorem{lemma}[theorem]{Lemma}

\theoremstyle{remark}

\numberwithin{equation}{section}

\newcommand{\Acal}{\mathscr{A}}

\newcommand{\Hcal}{\mathscr{H}}

\newcommand{\Lcal}{\mathscr{L}}

\newcommand{\Pcal}{\mathscr{P}}

\newcommand{\Z}{\mathbb{Z}}
\newcommand{\C}{\mathbb{C}}

  \DeclareFontFamily{U}{wncy}{}
    \DeclareFontShape{U}{wncy}{m}{n}{<->wncyr10}{}
    \DeclareSymbolFont{mcy}{U}{wncy}{m}{n}
    \DeclareMathSymbol{\Sha}{\mathord}{mcy}{"58}

\begin{document}
\title[]{Hilbert's tenth problem for lacunary entire functions of finite order}

\author{Natalia Garcia-Fritz}
\address{ Departamento de Matem\'aticas\newline
\indent Pontificia Universidad Cat\'olica de Chile\newline
\indent Facultad de Matem\'aticas\newline
\indent 4860 Av. Vicu\~na Mackenna\newline
\indent Macul, RM, Chile}
\email[N. Garcia-Fritz]{natalia.garcia@mat.uc.cl}

\author{Hector Pasten}
\address{ Departamento de Matem\'aticas\newline
\indent Pontificia Universidad Cat\'olica de Chile\newline
\indent Facultad de Matem\'aticas\newline
\indent 4860 Av. Vicu\~na Mackenna\newline
\indent Macul, RM, Chile}
\email[H. Pasten]{hpasten@gmail.com}%

\thanks{N.\ G.-F.\ was supported by ANID FONDECYT Regular grant 1211004 from Chile. H.P.\ was supported by ANID FONDECYT Regular grant 1230507 from Chile.}
\date{\today}
\dedicatory{To the memory of Thanases Pheidas}

\begin{abstract} In the context of Hilbert's tenth problem, an outstanding open case is that of complex entire functions in one variable. A negative solution is known for polynomials (by Denef) and for exponential polynomials of finite order (by Chompitaki, Garcia-Fritz, Pasten, Pheidas, and Vidaux), but no other case is known for rings of complex entire functions in one variable. We prove a negative solution to the analogue of Hilbert's tenth problem for rings of complex entire functions of finite order having lacunary power series expansion at the origin. 
\end{abstract}

\maketitle



\section{Introduction} 

\subsection{Hilbert's tenth problem} Hilbert's tenth problem asked for an algorithm to decide existence of solutions of polynomial equations over the integers. It was shown to be unsolvable by Matijasevich \cite{Matijasevic} after the work of Davis, Putnam, and Robinson \cite{DPR}; namely, the positive existential theory of the ring $\Z$ is undecidable.

After this negative solution to Hilbert's tenth problem, the analogous problem for other rings has been intensively studied. An outstanding case that remains open is that of the ring $\Hcal$ of complex entire functions in one variable $z$ seen as a structure over the language $\Lcal_z=\{0,1,z,+,\times,=\}$, see \cite{GhentSurvey}. The main difficulty is that, unlike the case of algebraic function fields, there are very large continuous families of transcendental holomorphic maps into group varieties. This fact prevents one from using the classical method of interpreting the integers using such varieties.

Nevertheless, there is considerable progress for other rings similar to  $\Hcal$;  the non-archimedean counterpart was negatively solved in \cite{LipPhe} and \cite{GarPas}, while the case of complex entire functions in at least two variables was negatively solved in \cite{PheVidHolo}. 

The current progress has been much more difficult for subrings of $\Hcal$ containing $\C[z]$, seen as $\Lcal_z$-structures. A negative solution of the analogue of Hilbert's tenth problem was first established  for complex polynomials by Denef \cite{Denef} and, more recently, for exponential polynomials of finite order by Chompitaki, Garcia-Fritz, Pasten, Pheidas, and Vidaux \cite{CGPPV}. This covers all known cases for subrings of $\Hcal$ containing $\C[z]$.

In this work we give a negative solution to the analogue of Hilbert's tenth problem for rings of lacunary entire functions of finite order. Let us introduce some notation before giving a precise formulation of our main results.

\subsection{Notation} 

We use Landau's $O$-notation, with subscripts to indicate the dependence of the implicit constant on other parameters. An entire function $f\in \Hcal$ is said to be of \emph{finite order} if there is a constant $A\ge 0$ (depending on $f$) such that for all $r>1$ we have 
$$
\log \max\{|f(z)| : |z|\le r\} = O_f(r^A).
$$

Entire functions of finite order form a ring that we denote by $\Hcal_{\rm fin}$.

Let  $F=\sum_{n\ge 0} a_n z^n\in \C[[z]]$ be a formal power series.  We say that $F$ is \emph{holonomic} if there is an integer $k\ge 1$ and polynomials $P_0,P_1,...,P_k\in \C[z]$ with $P_k\ne 0$ such that $F$ satisfies the differential equation
$$
P_k F^{(k)} + ...+ P_1F^{(1)} + P_0F=0.
$$
On the other hand, we say that $F$ is \emph{lacunary} if for every $M\ge 1$ there is $N$ such that $a_{N+j}=0$ for each $j=1,2,...,M$.  

A function in one complex variable which is holomorphic in a neighborhood of $z=0$ (for instance, an entire function) is said to be \emph{holonomic} (resp.\ \emph{lacunary}) if its power series expansion at $z=0$ is. We remark that polynomials are holonomic, lacunary, and of finite order. 

Finally, we note that if a ring of entire functions contains $\C[z]$, then it can be regarded as a structure over the language $\Lcal_z=\{0,1,z,+,\times, =\}$ in a natural way.

\subsection{Main results} Our first result is the following:

\begin{theorem}\label{ThmMainMain} Let $\Acal\subseteq \Hcal_{\rm fin}$ be a ring containing $\C[z]$ such that $\{f\in \Acal : f\mbox{ is holonomic}\}=\C[z]$. Then $\Z$ is positive existentially definable in $\Acal$ over the language $\Lcal_z$. In particular, the analogue of Hilbert's tenth problem for $\Acal$ is unsolvable; that is, the positive existential theory of $\Acal$ over $\Lcal_z$ is undecidable.
\end{theorem}

We remark that every exponential polynomial of finite order is holonomic, see Section \ref{SecHolonomic}. Thus, Theorem \ref{ThmMainMain} (when $\Acal\ne \C[z]$) gives a negative solution to the analogue of Hilbert's tenth problem in cases that are completely different to those covered in \cite{CGPPV}.

The only  lacunary and holonomic entire functions are polynomials, cf.\ Section \ref{SecHolonomic}. Thus we get

\begin{theorem}\label{ThmMain} Let $\Acal\subseteq \Hcal_{\rm fin}$ be a ring containing $\C[z]$ such that every $f\in \Acal$ is lacunary. Then $\Z$ is positive existentially definable in $\Acal$ over the language $\Lcal_z$.  In particular, the analogue of Hilbert's tenth problem for $\Acal$ is unsolvable; that is, the positive existential theory of $\Acal$ over $\Lcal_z$ is undecidable.
\end{theorem}

\subsection{Examples} Theorem \ref{ThmMainMain} is applicable in the case $\Acal=\C[z]$, which is Denef's theorem \cite{Denef}. Let us describe a different example that contains plenty of transcendental functions.

For $f\in \Hcal$ with power series expansion $f=\sum_{n\ge 0} a_n z^n$ we define $N_f(x) = \#\{n\le x : a_n\ne 0\}$. Let $\Acal_{0}$ be the set of all $f\in \Hcal_{\rm fin}$ satisfying that for every $\epsilon>0$ we have $N_f(x)=O_{f,\epsilon}(x^{\epsilon})$. For instance, the following is an element of $\Acal_0$ (for growth order, compare with the exponential function):
$$
f(z)=\sum_{n\ge 1} \frac{z^{n^n}}{(n^n)!}.
$$
Every $f\in \Acal_0$ is lacunary and moreover $\C[z]\subseteq \Acal_0$. Note that for every $f,g\in \Hcal$ we have $N_{f+g}(x)\le N_f(x)+N_g(x)$ and $N_{f\cdot g}(x)\le N_f(x)\cdot N_g(x)$, thus, $\Acal_0$ is a ring and Theorem \ref{ThmMain} is applicable in the case $\Acal=\Acal_0$.

\subsection{Functional Pell equations} The proof of Theorem \ref{ThmMainMain} is inspired in the ideas pioneered by Denef in the case of polynomials \cite{Denef} using the functional Pell equation $X^2- (z^2-1)Y^2=1$. The key technical input for Theorem \ref{ThmMain} is the following result on this Pell equation, which can be of independent interest.

\begin{theorem}\label{ThmPell} Let $f,g\in \Hcal_{\rm fin}$ be entire functions of finite order satisfying $f^2 - (z^2-1)g^2=1$. Then $f$ and $g$ are holonomic.
\end{theorem}

The proof builds on the theory of holonomic functions as well as preliminary results from \cite{CGPPV} regarding holomorphic solutions of the Pell equation. See Section \ref{SecPell}.


\section{Holonomic functions} \label{SecHolonomic}

A sequence of complex numbers $a_0,a_1,...$ is called \emph{holonomic} if there is an integer $k\ge 1$ and polynomials $p_0,p_1,...,p_k\in \C[z]$ with $p_k\ne 0$ such that for every $n\ge 0$ we have
$$
p_k(n)\cdot a_{n+k} + ... + p_1(n)\cdot a_{n+1} + p_0(n)\cdot a_n=0.
$$
A basic result on holonomic functions is the following, see Theorem 7.1 \cite{KauersPaule}.

\begin{lemma}[Relation with sequences] Let $(a_n)_{n\ge 0}$ be a sequence of complex numbers. Then $(a_n)_{n\ge 0}$ is holonomic if and only if the power series $\sum_{n\ge 0} a_nz^n$ is holonomic.
\end{lemma}

As an immediate consequence we get

\begin{corollary}\label{Coro} Let $F\in \C[[z]]$ be a lacunary power series. Then $F$ is holonomic if and only if it is a polynomial.
\end{corollary}

We will also need the following closure properties, that we only state in the case of holomorphic functions ---the case of formal power series also holds, with an appropriate definition of composition. See Theorem 7.2 in \cite{KauersPaule}.

\begin{lemma}[Closure properties]\label{Lemmah2} Let $f$ and $g$ be functions of one complex variable that are holomorphic in a neighborhood of $z=0$. 

\begin{itemize}
\item[(i)] If $f$ is algebraic over $\C(z)$, then it is holonomic.
\item[(ii)] If $f$ and $g$ are holonomic, then $f+g$ and $fg$ are holonomic.
\item[(iii)] If $f$ is algebraic over $\C(z)$ and $g$ is an entire holonomic function, then $g\circ f$ is holonomic.
\end{itemize}
\end{lemma}


\section{Pell equations} \label{SecPell}

As mentioned in the introduction, we will be interested in the holomorphic solutions of the Pell equation 
\begin{equation}\label{EqnPell}
X^2 - (z^2-1)Y^2=1.
\end{equation}

Let $w$ be the $2$-valued algebraic function defined by $w^2=z^2-1$. For $n\in \Z$ let us define $x_n,y_n\in \C[z]$ by the identity $x_n+wy_n = (z+w)^n$. With this notation, let us recall some preliminary results on the functional solutions of \eqref{EqnPell}.

\begin{lemma}[See Section 2 in \cite{Denef}] \label{Lemma1} Let $f,g\in \C[z]$. Then $X=f$, $Y=g$ is a solution of \eqref{EqnPell} if and only if there are $\epsilon\in\{-1,1\}$ and $n\in \Z$ with $(f,g) = (\epsilon x_n,  y_n)$.
Furthermore, $x_n,y_n\in \Z[z]$ and for all $n\in \Z$ we have $y_n(1)=n$.
\end{lemma}
\begin{lemma}[See Theorems 2.4 and 3.4 in \cite{CGPPV}]\label{Lemma2} Let $f,g\in \Hcal$. Then $X=f$, $Y=g$ is a solution of \eqref{EqnPell} if and only if there are $\epsilon\in\{-1,1\}$, $h\in \Hcal$ and $n\in \Z$ with $f+wg = \epsilon\cdot (z+w)^n\cdot \exp(wh)$. Furthermore, if these equivalent conditions hold and $f,g\in \Hcal_{\rm fin}$, then $h\in \C[z]$.
\end{lemma}

With this at hand, we can prove Theorem \ref{ThmPell}.

\begin{proof}[Proof of Theorem \ref{ThmPell}] Since $f,g\in  \Hcal_{\rm fin}$,  Lemma \ref{Lemma2} shows that  there are $\epsilon\in\{-1,1\}$, $h\in \C[z]$, and $n\in \Z$ with $f+wg = \epsilon\cdot (z+w)^n\cdot \exp(wh)$. Note that in a neighborhood of $z=0$, $w$ has two holomorphic branches $W$ and $-W$ determined by $W(0)=i$ and we obtain the following identities of holomorphic functions near $z=0$:
$$
f+Wg = \epsilon\cdot (z+W)^n\cdot \exp(Wh)\mbox{ and }f-Wg = \epsilon\cdot (z-W)^n\cdot \exp(-Wh).
$$

Noticing that $Wh$ is algebraic over $\C(z)$, Lemma \ref{Lemmah2} applied to the right hand side of these expressions gives that $A=f+Wg$ and $B=f-Wg$ are holonomic. Since $1/W$ is also algebraic and holomorphic near $z=0$, it is holonomic. Thus, $f=(A+B)/2$ and $g=(A-B)/(2W)$ are holonomic.
\end{proof}


\section{Definability of the integers} \label{SecDef}

\begin{proof}[Proof of Theorem \ref{ThmMainMain}] The argument is standard and we include it for the sake of completeness.

By Picard's theorem, the positive existential (p.e.) $\Lcal_z$-formula 
$$
c(x):\quad \exists y, y^2=x^5-1
$$
defines $\C$ in $\Acal$ because $y^2=x^5-1$ defines a curve of geometric genus $2$. By Theorem \ref{ThmPell}, the condition $\{f\in \Acal : f\mbox{ is holonomic}\}=\C[z]$, and Lemma \ref{Lemma1}, the p.e.\ $\Lcal_z$-formula
$$
p(y): \quad \exists x, x^2 - (z^2-1)y^2=1
$$
defines the set $\Pcal = \{y_n : n\in \Z\}\subseteq \Z[z]$. 

Note that since $\C[z]\subseteq \Acal$, for $\lambda\in \C$ and $y\in \C[z]$ the condition $y(1)=\lambda$ is equivalent to the  condition: there is $ f\in \Acal$ such that $(z-1)f = y-\lambda$.

Therefore, the p.e.\ $\Lcal_z$-formula
$$
\xi(t):\quad c(t)\wedge\exists y\exists f,  p(y)\wedge  (z-1)f=y-t.
$$
defines $\{y(1) : y\in \Pcal\}$ which is $\Z$ by Lemma \ref{Lemma1}. The undecidability result follows, since the p.e.\ theory of the ring $\Z$ is undecidable \cite{DPR, Matijasevic}.
\end{proof}

\begin{proof}[Proof of Theorem \ref{ThmMain}] By Theorem \ref{ThmMainMain} and Corollary \ref{Coro}.
\end{proof}


\section{Acknowledgments}

N. G.-F. was supported by ANID FONDECYT Regular grant 1211004 from Chile. H. P. was supported by ANID FONDECYT Regular grant 1230507 from Chile.

We thank Xavier Vidaux for several comments on a first version of this manuscript.

This research was heavily influenced by the ideas and work of Thanases Pheidas, who was one of the driving forces in the study of extensions of Hilbert's tenth problem, especially in the case of rings and fields of functions; see for instance \cite{CGPPV, LipPhe, PheI, PheII, PheInv,  PheVidHolo, GhentSurvey} among other works. We respectfully dedicate this article to his memory and legacy. He will be dearly missed.



\begin{thebibliography}{9}         

\bibitem{CGPPV}  D. Chompitaki, N. Garcia-Fritz, H. Pasten, T. Pheidas, X. Vidaux, \emph{The Diophantine problem for rings of exponential polynomials}. Ann. Sc. Norm. Super. Pisa Cl. Sci. (5) 23 (2022), no. 4, 1625--1636. 

\bibitem{DPR} M. Davis, H. Putnam, J. Robinson, \emph{The decision problem for exponential diophantine equations}. Ann. of Math. (2) 74 (1961), 425--436.

\bibitem{Denef} J. Denef, \emph{The diophantine problem for polynomial rings and fields of rational functions}, Trans. Amer. Math. Soc. {242} (1978), 391--399.

\bibitem{GarPas} N. Garcia-Fritz, H. Pasten, \emph{Uniform positive existential interpretation of the integers in rings of entire functions of positive characteristic}, J. Number Theory {156} (2015), 368--393.

\bibitem{KauersPaule} M. Kauers, P. Paule, \emph{The concrete tetrahedron. Symbolic sums, recurrence equations, generating functions, asymptotic estimates}. Texts and Monographs in Symbolic Computation. Springer Wien New York, Vienna, 2011.

\bibitem{LipPhe} L. Lipshitz, T. Pheidas, \emph{An analogue of Hilbert's Tenth Problem for p-adic entire functions}, J. Symbolic Logic {60-4} (1995), no. 4, 1301--1309.

\bibitem{Matijasevic}   Ju. V. Matijasevich, \emph{The Diophantineness of enumerable sets}, (Russian) Dokl. Akad. Nauk SSSR {191} (1970), 279--282. 

\bibitem{PheI} T. Pheidas \emph{An undecidability result for power series rings of positive characteristic}. Proc. Amer. Math. Soc. 99 (1987), no. 2, 364--366.

\bibitem{PheII} T. Pheidas, \emph{An undecidability result for power series rings of positive characteristic. II}. Proc. Amer. Math. Soc. 100 (1987), no. 3, 526--530.

\bibitem{PheInv} T. Pheidas, \emph{Hilbert's tenth problem for fields of rational functions over finite fields}. Invent. Math. 103 (1991), no. 1, 1--8. 

\bibitem{PheVidHolo} T. Pheidas, X. Vidaux, \emph{Hilbert's Tenth Problem for complex meromorphic functions in several variables}. Int. Math. Res. Not. IMRN 2022, no. 19, 15474--15504.

\bibitem{GhentSurvey} T. Pheidas, K. Zahidi, \emph{Undecidability of existential theories of rings and fields: A survey}, Contemp. Math. {270} (2000), 49--106.


\end{thebibliography}
\end{document}